\documentclass{amsart}
\usepackage{amssymb}
\usepackage{amsmath}
\usepackage{amsfonts}
\usepackage{mathrsfs}
\usepackage{hyperref}

\newtheorem{theorem}{Theorem}[section]
\newtheorem{lemma}[theorem]{Lemma}

\theoremstyle{definition}

\theoremstyle{remark}

\numberwithin{equation}{section}

\begin{document}
\title[BSE Property for Weighted Dirichlet Series]{The BSE Property for Weighted Banach Algebras of Dirichlet Series}
\author[P. A. Dabhi]{Prakash A. Dabhi}
\address{Institute of Infrastructure Technology Research and Management (IITRAM), Ahmedabad - 380026, Gujarat, India}
\email{lightatinfinite@gmail.com, prakashdabhi@iitram.ac.in}
\thanks{The author is grateful to the National Board for Higher Mathematics (NBHM), India, for the research grant (02011/39/2025/NBHM(R. P.)/R \& D II/16090).}
\subjclass[2020]{Primary 46J05, 43A20, 30B50; Secondary 43A22, 46B10}
\keywords{BSE-algebra, Dirichlet series, weight, semicharacter, Gel'fand space, $L$-projection, bidual space}
\begin{abstract}
Let $\omega: \mathbb{N} \to [1, \infty)$ be a weight on the multiplicative semigroup of natural numbers. We investigate the Bochner-Schoenberg-Eberlein (BSE) property for the weighted Banach algebra of Dirichlet series $A_\omega = \ell^1(\mathbb{N}, \omega)$ under Dirichlet convolution. By employing Bohr's correspondence, we establish a weak-$\ast$ approximation framework for $\omega$-bounded semicharacters using scaled truncations that lie within the predual space $c_0(\mathbb{N}, 1/\omega)$. Utilizing the canonical $L$-projection decomposition of the bidual space $A_\omega^{\ast\ast} = A_\omega \oplus c_0(\mathbb{N}, 1/\omega)^\perp$ alongside the weak-$\ast$ compactness of the unit ball via the Banach-Alaoglu theorem, we prove that $A_\omega$ is a BSE-algebra for any weight. Furthermore, we show that the intrinsic BSE-norm on the Gel'fand transforms is isometric to the natural weighted algebra norm $\|\cdot\|_\omega$.
\end{abstract}

\maketitle

\section{Introduction and Preliminaries}

The characterization of functions that can be represented as the transform of an element within a given Banach algebra is a foundational cornerstone of harmonic analysis. The classical prototype of this inquiry traces back to the Bochner-Schoenberg-Eberlein (BSE) theorem, established sequentially by Bochner \cite{bochner}, Schoenberg \cite{schoenberg}, and Eberlein \cite{eberlein}. In its original setting, the theorem establishes that a bounded continuous function $\sigma$ on the dual group $\widehat{G}$ of a locally compact abelian group $G$ is the Fourier-Stieltjes transform of a bounded complex Radon measure on $G$ if and only if it can be extended as a bounded linear functional on the span of $\widehat G$ with the dual norm of $L^1(G)^\ast$.

Takahasi and Hatori \cite{takahasi1990} lifted this classical property to the abstract setting of commutative semisimple Banach algebras, introducing the notion of a \emph{BSE-algebra}. In recent years, the BSE property has attracted significant attention across diverse functional analytic frameworks. For instance, Kaniuth and \"Ulger \cite{kaniuth2010} comprehensively investigated the BSE property for Fourier and Segal algebras, while Hatori and Takahasi \cite{hatori2001} explored generalized BSE-inequalities. Concurrently, recent investigations have actively extended these properties to arithmetic semigroup algebras such as $\ell^1(\mathbb N_{\gcd})$ and $\ell^1(\mathbb N_{\text{lcm}})$ \cite{dabhi2026}. While the BSE property has been extensively mapped for these group and semigroup structures, its extension to spaces of Dirichlet series provides a distinct analytical challenge. The study of Dirichlet series via modern Banach algebra techniques bridges classical number theory with infinite-dimensional holomorphicity \cite{bayart, helson}. By determining whether a weighted algebra of formal Dirichlet series $A_\omega = \ell^1(\mathbb N, \omega)$ is a BSE-algebra, we establish how closely the topological dual forces the global analytic functions on its infinite-dimensional Gel'fand space to reflect elements of the sequence space itself. Furthermore, we prove a strict isometric equivalence between the algebra norm and the intrinsic BSE-norm.

Let $A$ be a commutative Banach algebra. A nonzero linear functional $\varphi:A \to \mathbb C$ is a complex homomorphism if $\varphi(ab)=\varphi(a)\varphi(b)$ for all $a, b \in A$. Let $\Phi_A$ be the collection of all complex homomorphisms on $A$. For $a \in A$, let $\widehat a:\Phi_A \to \mathbb C$ be defined by $\widehat a(\varphi)=\varphi(a)$ for all $\varphi \in \Phi_A$. The map $\widehat a$ is the \emph{Gel'fand transform} of $a$. Let $\widehat A=\{\widehat a:a\in A\}$. The weakest topology on $\Phi_A$ making each element of $\widehat A$ continuous is the \emph{Gel'fand topology} on $\Phi_A$, and $\Phi_A$ equipped with the Gel'fand topology is the \emph{Gel'fand space} of $A$. The Gel'fand space $\Phi_A$ of $A$ is a locally compact Hausdorff space and it is compact when $A$ is unital. A commutative Banach algebra $A$ is \emph{semisimple} if $\bigcap_{\varphi \in \Phi_A} \ker\varphi=\{0\}$, where $\ker\varphi=\{a \in A:\varphi(a)=0\}$.

Let $A$ be a commutative Banach algebra without order, i.e., if $aA=\{0\}$, then $a=0$. A map $T:A \to A$ is a \emph{multiplier} \cite{larsen1971} on $A$ if $T(ab)=aTb=(Ta)b$ for all $a, b \in A$. Let $\mathcal M(A)$ be the collection of all multipliers on $A$. Then, as detailed by Kaniuth \cite[Proposition 1.4.11]{kaniuth2009}, $\mathcal M(A)$ is a unital commutative Banach algebra with the operator norm $\|T\|=\sup\{\|Ta\|:a \in A, \|a\|\leq 1\}$ for all $T\in \mathcal M(A)$. The Banach algebra $A$ is embedded in $\mathcal M(A)$ via the map $a\mapsto L_a$, where $L_a(b)=ab$ for all $b \in A$. Also, $A$ is unital if and only if $\mathcal M(A)=A$. By \cite[Theorem 1.2.2]{larsen1971}, $T \in \mathcal M(A)$ if and only if there is a unique bounded continuous function $\widehat T:\Phi_A\to \mathbb C$ such that $\widehat{Ta}(\varphi)=\widehat T(\varphi)\widehat a(\varphi)$ for all $a \in A$ and $\varphi \in \Phi_A$. Let $\widehat{\mathcal M(A)}=\{\widehat T:T \in \mathcal M(A)\}$. Note that if $A$ is semisimple, then $A$ is without order.

Let $A$ be a commutative Banach algebra without order. A continuous function $\sigma:\Phi_A\to \mathbb C$ is a \emph{BSE-function} \cite{takahasi1990} if there is a positive constant $C$ such that
\begin{equation}\label{bse_ineq}
\left|\sum_{i=1}^n c_i\sigma(\varphi_i)\right|\leq C\left\|\sum_{i=1}^n c_i\varphi_i\right\|_{A^\ast},
\end{equation}
whenever $n \in \mathbb N$, $\varphi_1, \varphi_2,\ldots, \varphi_n$ are in $\Phi_A$ and $c_1,c_2,\ldots,c_n \in \mathbb C$. The smallest such $C$ is the \emph{BSE-norm}, $\|\sigma\|_{\operatorname{BSE}}$, of $\sigma$. Let $C_{\operatorname{BSE}}(\Phi_A)$ be the collection of all BSE-functions. Then, by Takahasi and Hatori \cite{takahasi1990}, $(C_{\operatorname{BSE}}(\Phi_A),\|\cdot\|_{\operatorname{BSE}})$ is a semisimple commutative Banach algebra. The Banach algebra $A$ is a \emph{BSE-algebra} if $C_{\operatorname{BSE}}(\Phi_A)=\widehat{\mathcal M(A)}$. The acronym BSE stands for \emph{Bochner-Schoenberg-Eberlein}. Bochner, in \cite{bochner}, established that the group algebra $L^1(\mathbb R)$ is a BSE-algebra, and Schoenberg \cite{schoenberg} and Eberlein \cite{eberlein} subsequently extended this framework to any locally compact abelian group $G$, proving that $L^1(G)$ is a BSE-algebra.

The incorporation of a weight is a vital tool in modern harmonic analysis. As initially pioneered by Beurling \cite{beurling1938} for group algebras, weights provide a precise mechanism to control the decay rates of coefficients and directly govern the analytic regularity of the corresponding transforms. The critical role of the weight was recently emphasized in \cite{dabhi_beurling}, where it is established that the BSE and BED properties for the Beurling algebra $L^1(G,\omega)$ rely on the specific structural and bounding conditions imposed on $\omega$. In the specific setting of Dirichlet series, the chosen weight strictly determines the abscissa of absolute convergence and dictates the geometric boundaries of the underlying Gel'fand space. Let $\mathbb H$ be the closed right half plane, defined as $\mathbb H=\{x+iy \in \mathbb C:x\geq 0\}$. Let $\omega:\mathbb N \to [1,\infty)$ be a \emph{weight}, i.e., $\omega(mn)\leq \omega(m)\omega(n)$ for all $m,n \in \mathbb N$. Let $A_\omega$ be the collection of all functions $f:\mathbb H \to \mathbb C$ such that $f$ can be represented as $f(s)=\sum_{n=1}^\infty \frac{a(n)}{n^s}$ for all $s \in \mathbb H$ and $\|f\|_\omega=\sum_{n=1}^\infty |a(n)|\omega(n)<\infty$. Then $(A_\omega,\|\cdot\|_\omega)$ is a unital commutative Banach algebra with this norm and the usual pointwise multiplication of functions. Since $A_\omega$ is unital, $\mathcal M(A_\omega)=A_\omega$. Thus, $A_\omega$ is a BSE-algebra if and only if $C_{\operatorname{BSE}}(\Phi_{A_\omega})=\widehat{A_\omega}$. The space $A_\omega$ is the \emph{weighted Banach algebra of Dirichlet series}.

Consider the usual multiplication on $\mathbb N$ as a binary operation. Let $\omega$ be a weight on $\mathbb N$. Let $$\ell^1(\mathbb N,\omega)=\left\{f=(a(n)):\mathbb N \to \mathbb C:\|f\|_\omega=\sum_{n=1}^\infty|a(n)|\omega(n)<\infty\right\}.$$For $f=(a(n))$ and $g=(b(n))$ in $\ell^1(\mathbb N,\omega)$, define their \emph{(Dirichlet) convolution} $f\star g$ by
\begin{equation}\label{convolution}
(f\star g)(n)=\sum_{uv=n}f(u)g(v)=\sum_{d\mid n}f(d)g\left(\frac{n}{d}\right)\quad(n\in \mathbb N).
\end{equation}
Then $(\ell^1(\mathbb N,\omega),\|\cdot\|_\omega)$ is a unital commutative Banach algebra with the above convolution as multiplication. The Banach algebras $A_\omega$, defined above, and $\ell^1(\mathbb N,\omega)$ are isometrically isomorphic via a map $f=(a(n))\mapsto f(s)=\sum_{n=1}^\infty \frac{a(n)}{n^s}$ from $\ell^1(\mathbb N,\omega)$ onto $A_\omega$.

A nonzero map $\chi:\mathbb N \to \mathbb C$ is an $\omega$-bounded semicharacter if $\chi(mn)=\chi(m)\chi(n)$ and $|\chi(n)|\leq \omega(n)$ for all $m, n \in \mathbb N$. Let $\widehat{\mathbb N_\omega}$ be the collection of all $\omega$-bounded semicharacters on $\mathbb N$. By standard semigroup algebra theory \cite{hewitt1956}, a map $\varphi:A_\omega\to \mathbb C$ is a complex homomorphism if and only if there is a unique $\chi \in \widehat{\mathbb N_\omega}$ such that $\varphi=\varphi_\chi$, where $$\varphi_\chi(f)=\sum_{n=1}^\infty a(n)\chi(n)\quad(f(s)=\sum_{n=1}^\infty \frac{a(n)}{n^s}\in A_\omega).$$ Thus there is a one-to-one correspondence between $\widehat{\mathbb N_\omega}$ and $\Phi_{A_\omega}$ given by $\chi\mapsto \varphi_\chi$. We shall not differentiate between $\chi$ and $\varphi_\chi$. By Bohr's correspondence \cite{bohr1913}, any $\omega$-bounded semicharacter $\chi$ is uniquely determined by its values on the prime numbers. This allows $\widehat{\mathbb N_\omega}=\Phi_{A_\omega}$ to be topologically identified with a compact subset of $\mathbb C^\infty$. Because $\omega(n) \geq 1$ for all $n \in \mathbb N$, we have a continuous inclusion $A_\omega \subseteq \ell^1(\mathbb N)$. This implies that the Gel'fand space $\Phi_{A_\omega}$ of $A_\omega$ contains the standard infinite closed polydisc $\mathbb D^\infty$, the Gel'fand space of $\ell^1(\mathbb N)=A_1$, where $\mathbb D=\{z\in \mathbb C:|z|\leq 1\}$. For our scaling arguments, we also utilize the arithmetic function $\Omega(n)$, which denotes the total number of prime factors of an integer $n$ counted with exact multiplicity.

The dual space of $A_\omega$, denoted by $A_\omega^\ast$, is isometrically isomorphic to the weighted supremum space $\ell^\infty(1/\omega)$, where $\ell^\infty(1/\omega)=\{x=(x(n)):x(n)\in \mathbb C, \|x\|=\sup_n\frac{|x(n)|}{\omega(n)}<\infty\}$. The norm of any continuous linear functional $\Psi \in A_\omega^\ast$ is evaluated by the supremum over the canonical basis elements, giving $\|\Psi\|_{A_\omega^\ast} = \sup_{n \in \mathbb N} |\Psi(n)| / \omega(n)$. Furthermore, $A_\omega$ is the dual space of the weighted null sequence space $c_0(\mathbb N, 1/\omega)$. Here $c_0(\mathbb N,1/\omega)$ is the Banach space of all complex sequences $x=(x(n))$ such that $\frac{x(n)}{\omega(n)}\to 0$ as $n\to \infty$ and with the norm $\|x\|=\sup_n\frac{|x(n)|}{\omega(n)}$. By standard Banach space duality, detailed comprehensively by Dales \cite{dales2000}, the bidual space $A_\omega^{\ast\ast}$ admits a canonical $L$-projection decomposition $A_\omega^{\ast\ast} = A_\omega \oplus c_0(\mathbb N, 1/\omega)^\perp$. For any functional $\Phi \in A_\omega^{\ast\ast}$, this decomposition uniquely separates it into an absolutely summable sequence component $f \in A_\omega$ and a singular component $\Phi_{\text{sing}}$ that vanishes entirely on $c_0(\mathbb N, 1/\omega)$. Crucially, the norm satisfies $\|\Phi\| = \|f\|_\omega + \|\Phi_{\text{sing}}\|$, a property that directly facilitates the unconditional bounding in our main results.

\section{Main Results and Proof}
Let $\{p_1,p_2,\ldots\}$ be the collection of all primes such that $p_1<p_2<\cdots$.
\begin{lemma}\label{lem:character_approx}
Let $\omega$ be a weight on $\mathbb N$. Fix $\varphi \in \Phi_{A_\omega}$, $k \in \mathbb N$ and $0 < r < 1$. Define the sequence $Z_{k,r} \in \mathbb C^{\mathbb N}$ by its values on primes as follows: $Z_{k,r}(p_i) = r \varphi(p_i)$ for $i \leq k$, and $Z_{k,r}(p_i) = 0$ for $i > k$. Extend it to all $n \in \mathbb N$ so that it is multiplicative. Define $\varphi_k \in \Phi_{A_\omega}$ to be the semicharacter matching $\varphi$ on the first $k$ primes and vanishing on all other primes. Then the following statements hold.
\begin{enumerate}
    \item The sequence $Z_{k,r}$ belongs to $\Phi_{A_\omega} \cap c_0(\mathbb N, 1/\omega)$.
    \item $Z_{k,r}\to\varphi_k$ in the weak-$\ast$ topology as $r \nearrow 1$.
    \item $\varphi_k\to\varphi$ in the weak-$\ast$ topology as $k \to \infty$.
\end{enumerate}
\end{lemma}

\begin{proof}
For the first claim, the multiplicative extension of $Z_{k,r}$ means that if $n=p_1^{\alpha_1}\cdots p_k^{\alpha_k}$ for some nonnegative integers $\alpha_1, \ldots, \alpha_k$, then we take $Z_{k,r}(n) = r^{\alpha_1+\cdots+\alpha_k}\varphi(n) = r^{\Omega(n)}\varphi(n)$ and if $p_j|n$ for some $j>k$, then $Z_{k,r}(n)=0$. Since $\varphi \in \Phi_{A_\omega}$, $|\varphi(n)| \leq \omega(n)$ for all $n\in \mathbb N$. As $r < 1$, we have $|Z_{k,r}(n)| \leq r^{\Omega(n)} \omega(n)$ for all $n \in \mathbb N$. As $n \to \infty$ within the support generated by the first $k$ primes, the exponent $\Omega(n) \to \infty$. As $0<r<1$, the ratio $Z_{k,r}(n) / \omega(n)\to 0$ as $n\to \infty$. It means that $Z_{k,r}$ is in $c_0(\mathbb N, 1/\omega)$. Because $|Z_{k,r}(n)| \leq \omega(n)$ for all $n$ and $Z_{k,r}$ is multiplicative, $Z_{k,r}$ is an $\omega$-bounded semicharacter, i.e., $Z_{k,r}\in \Phi_{A_\omega}$.

To prove the second claim, let $S_k$ denote the subsemigroup generated by the first $k$ primes. Let $g(s) = \sum_{n=1}^\infty a(n) n^{-s}$ be in $A_\omega$. Then
$$|Z_{k,r}(g) - \varphi_k(g)| \leq \sum_{n \in S_k} |a(n)| |\varphi(n)| (1 - r^{\Omega(n)}).$$
Because $|\varphi(n)| \leq \omega(n)$ for all $n$, the individual terms in this sum are bounded by $|a(n)| \omega(n)$. Since $g \in A_\omega$, $\sum_{n=1}^\infty |a(n)| \omega(n)<\infty$. For every fixed $n \in S_k$, the factor $(1 - r^{\Omega(n)})$ goes to zero as $r \nearrow 1$. It follows using the dominated convergence theorem that $\lim_{r\nearrow 1}\sum_{n \in S_k} |a(n)| |\varphi(n)| (1 - r^{\Omega(n)})=0$. This proves that $Z_{k,r}(g) \to \varphi_k(g)$ as $r\nearrow 1$. It means that $Z_{k,r}$ converges in the weak-$\ast$ topology to the semicharacter $\varphi_k$ as $r\nearrow 1$.

For the third claim, we take the limit as $k \to \infty$. For the same $g(s) \in A_\omega$, we have
$$|\varphi(g) - \varphi_k(g)| \leq \sum_{n \notin S_k} |a(n)| \omega(n).$$
Because $\|g\|_\omega=\sum_{n=1}^\infty|a(n)|\omega(n)<\infty$, the tail sum $\sum_{n \notin S_k} |a(n)| \omega(n) \to 0$ as $k \to \infty$. It means that $\varphi_k(g) \to \varphi(g)$ as $k\to \infty$. Therefore the sequence $(\varphi_k)$ converges to $\varphi$ in the weak-$\ast$ topology.
\end{proof}

\begin{theorem}\label{main}
For any weight $\omega: \mathbb N \to [1, \infty)$, the Banach algebra $A_\omega$ is a $\operatorname{BSE}$-algebra.
\end{theorem}
\begin{proof}
Since $A_\omega$ is unital, by \cite[Corollary 5]{takahasi1990}, $\widehat{\mathcal M(A_\omega)}\subset C_{\operatorname{BSE}}(\Phi_{A_\omega})$. For the reverse inclusion, take any $\sigma \in C_{\operatorname{BSE}}(\Phi_{A_\omega})$. Then, by \cite[Theorem 4(i)]{takahasi1990}, there exists a net $(f_\alpha)$ in $A_\omega$ such that $\|f_\alpha\|_\omega \leq \|\sigma\|_{\operatorname{BSE}}$ and $\widehat{f_\alpha}(\varphi) \to \sigma(\varphi)$ for all $\varphi \in \Phi_{A_\omega}$. By the Banach-Alaoglu theorem, the bounded net $(f_\alpha)$ possesses a weak-$*$ cluster point $\Phi$ in the bidual $A_\omega^{**} = \ell^\infty(\mathbb N, 1/\omega)^*$. By passing to a suitable subnet, we may assume without loss of generality that the net $(f_\alpha)$ itself converges to $\Phi$ in the weak-$*$ topology. The weak-$*$ lower semicontinuity of the norm guarantees $\|\Phi\| \leq \liminf \|f_\alpha\|_\omega \leq \|\sigma\|_{\operatorname{BSE}}$. Furthermore, because every semicharacter $\varphi \in \Phi_{A_\omega}$ is an element of the dual space $A_\omega^*$, this weak-$*$ convergence implies that $\Phi(\varphi) = \lim_\alpha \widehat{f_\alpha}(\varphi)$. Since the original net was defined such that its Gel'fand transform converges pointwise to $\sigma$, the subnet inherits this exact limit, yielding $\Phi(\varphi) = \sigma(\varphi)$.

Using the canonical $L$-projection decomposition $A_\omega^{**} = A_\omega \oplus c_0(\mathbb N, 1/\omega)^\perp$, we can uniquely write $\Phi = f + \Phi_{\text{sing}}$, where $f \in A_\omega$ is the absolutely summable part and $\Phi_{\text{sing}}$ is a functional vanishing entirely on $c_0(\mathbb N, 1/\omega)$. Because this decomposition is an $L$-projection, the norm satisfies $\|\Phi\| = \|f\|_\omega + \|\Phi_{\text{sing}}\|$. This gives $\|f\|_\omega \leq \|\Phi\| \leq \|\sigma\|_{\operatorname{BSE}}$.

We now show that $\widehat{f} = \sigma$. Fix an arbitrary semicharacter $\varphi \in \Phi_{A_\omega}$. For any $k \in \mathbb N$ and $0 < r < 1$, construct the sequences $Z_{k,r}$ and $\varphi_k$ as defined in Lemma \ref{lem:character_approx}. By the first claim of the lemma, $Z_{k,r} \in \Phi_{A_\omega}$, which implies $\sigma(Z_{k,r}) = \Phi(Z_{k,r})$. Since $Z_{k,r}$ also belongs to $c_0(\mathbb N, 1/\omega)$, the singular component $\Phi_{\text{sing}}$ vanishes on $Z_{k,r}$. This gives $$\sigma(Z_{k,r}) = \Phi(Z_{k,r}) = f(Z_{k,r}) + \Phi_{\text{sing}}(Z_{k,r}) = \widehat{f}(Z_{k,r}).$$

We establish the final equality via continuity. By the second claim of Lemma \ref{lem:character_approx}, $Z_{k,r} \to \varphi_k$ in the weak-$\ast$ topology as $r \nearrow 1$. The weak-$\ast$ continuity of $\sigma$ and $\widehat{f}$ implies $\sigma(\varphi_k) = \widehat{f}(\varphi_k)$ for all $k \in \mathbb N$. Finally, by the third claim of Lemma \ref{lem:character_approx}, $\varphi_k \to \varphi$ in the weak-$\ast$ topology as $k \to \infty$. Applying the continuity of $\sigma$ and $\widehat{f}$ once more gives
$$\sigma(\varphi) = \lim_{k \to \infty} \sigma(\varphi_k) = \lim_{k \to \infty} \widehat{f}(\varphi_k) = \widehat{f}(\varphi).$$
This proves that $C_{\operatorname{BSE}}(\Phi_{A_\omega}) = \widehat{A_\omega}$.
\end{proof}

The following theorem establishes that the intrinsic BSE-norm is isometrically equivalent to the algebra norm.

\begin{theorem}
If $\omega$ is a weight on $\mathbb{N}$, then $\|\widehat{f}\|_{\operatorname{BSE}} = \|f\|_\omega$ for all $f \in A_\omega$.
\end{theorem}

\begin{proof}
Let $f \in A_\omega$. The upper bound $\|\widehat{f}\|_{\operatorname{BSE}} \leq \|f\|_\omega$ follows from the definition of the BSE inequality. Conversely, let $\sigma = \widehat{f} \in C_{\operatorname{BSE}}(\Phi_{A_\omega})$. In the proof of the preceding main theorem, we proved that any such function originates from an element $f \in A_\omega$ constructed via the $L$-projection of a weak-$\ast$ cluster point $\Phi$ in the bidual. During that construction, we established $\|f\|_\omega \leq \|\Phi\| \leq \|\sigma\|_{\operatorname{BSE}}$. Since $\sigma = \widehat{f}$, $\|f\|_\omega \leq \|\widehat{f}\|_{\operatorname{BSE}}$.
\end{proof}

\end{document}